\documentclass[10pt]{amsart}

\usepackage[latin1]{inputenc}
\usepackage{amsmath}
\usepackage{amsfonts}
\usepackage{amssymb}
\usepackage{ytableau}
\usepackage[mathscr]{eucal}
\usepackage{diagrams}
\usepackage{xy,color,graphicx}
\usepackage{hyperref}
\xyoption{all}
\usepackage{array}
\usepackage{enumerate}
\usepackage{url}
\usepackage{halloweenmath}

\newcommand{\N}{\mathbb{N}}
\newcommand{\Q}{\mathbb{Q}}
\newcommand{\C}{\mathbb{C}}
\newcommand{\Z}{\mathbb{Z}}

\newcommand{\vtx}[1]{*+[o][F-]{\scriptscriptstyle #1}}

\newcommand{\F}{\mathbb{F}}

\newtheorem{definition}{Definition}
\newtheorem{proposition}{Proposition}

\newtheorem{example}{Example}

\newtheorem{remark}{Remark}

\title[Motivic measures and $\F_1$-geometries]{Motivic measures and $\F_1$-geometries}
\author{Lieven Le Bruyn} 
\address{Department Mathematics, University  of Antwerp ,
 Middelheimlaan 1, B-2020 Antwerp (Belgium) {\tt lieven.lebruyn@uantwerpen.be}}

\begin{document}
\sloppy

\maketitle

\begin{abstract} Right adjoints for the forgetful functors on $\lambda$-rings and bi-rings are applied to motivic measures and their zeta functions on the Grothendieck ring of $\F_1$-varieties in the sense of Lorscheid and Lopez-Pena (torified schemes). This leads us to a specific subring of $\mathbb{W}(\Z)$, properly containing Almkvist's ring $\mathbb{W}_0(\Z)$, which  might be a natural  receptacle for all local factors of completed zeta functions.
\end{abstract}

\section{Introduction}

In \cite{Borger1} Jim Borger proposes to consider integral $\lambda$-rings as $\F_1$-algebras, with the $\lambda$-structure viewed as the descent data from $\Z$ to $\F_1$. Crucial is the fact that the functor of forgetting the $\lambda$-structure has the Witt-ring functor $\mathbb{W}(-)$ as its right adjoint. 

Recall that the $\lambda$-ring $\mathbb{W}(\Z) = 1 + t \Z [[t ]]$ has addition ordinary multiplication of power series, and a new multiplication induced functorially by demanding that $(1-mt)^{-1} \ast (1-nt)^{-1} = (1-mn t)^{-1}$. We will view $\mathbb{W}(\Z)$ as a  receptacle for motivic data, such as zeta-functions.

\vskip 3mm

A counting measure is a ringmorphism $\mu : K_0(\mathbf{Var}_{\Z}) \rTo \Z$, with $K_0(\mathbf{Var}_{\Z})$  the Grothendieck ring of schemes of finite type over $\Z$. A classic example being $\mu_{\F_p}([X]) = \# \overline{X}_p(\F_p)$ where $\overline{X}_p$ is the reduction of $X$ modulo $p$. The $\F_p$-counting measure $\mu_{\F_p}$ is {\em exponentiable} meaning that it defines a ringmorphism
\[
\zeta_{\F_p} : K_0(\mathbf{Var}_{\Z}) \rTo \mathbb{W}(\Z) \qquad [X] \mapsto \zeta_{\F_p}(\overline{X}_p,t) = exp(\sum_{r \geq 1} \# \overline{X}_p(\F_{p^r}) \frac{t^r}{r}) \]
and is {\em rational}, meaning that $\zeta_{\F_q}$ factors through the Almkvist subring $\mathbb{W}_0(\Z)$ of $\mathbb{W}(\Z)$, consisting of all rational functions.

\vskip 3mm

For a scheme $X$ of finite type over $\Z$, let  $N(x)$ for every closed point  $x \in | X |$  be the cardinality of the finite residue field at $x$, then the {\em Hasse-Weil zeta function} of $X$ decomposes as a product
\[
\zeta_X(s) = \prod_{x \in | X |} \frac{1}{(1-N(x)^{-s})} = \prod_p \zeta_{\F_p}(\overline{X}_p,p^{-s}) \]
over the non-archimedean local factors. If we take the product with the archimedean factors ($\Gamma$-factors) we obtain the completed zeta function $\hat{\zeta}_X(s)$.

\vskip 3mm

One of the original motivations for constructing $\F_1$-geometries was to understand these $\Gamma$-factors, see the lecture notes \cite{ManinZeta2} by Yuri I. Manin. For example, Manin conjectured that  Deninger's $\Gamma$-factor $\prod_{n \geq 0} \tfrac{s-n}{2 \pi}$ of $\overline{\mathbf{Spec}(\Z)}$ at complex infinity should be the zeta function of (the dual of) infinite dimensional projective space $\mathbb{P}^{\infty}_{\F_1}$, see \cite[4.3]{ManinNumbers} and \cite[Intro]{ManinZeta}. 

\vskip 3mm

As a step towards this conjecture, we proposed in \cite{LBrecursive} to consider integral bi-rings as $\F_1$-algebras, this time with the co-ring structure as the descent data from $\Z$ to $\F_1$. Here again, the forgetful functor has a right adjoint with assigns to $\Z$ the bi-ring $\mathbb{L}(\Z)$ of all integral recursive sequences equipped with the Hadamard product. These two approaches to $\F_1$-geometry are related, that is, we have a commuting diagram of (solid) ringmorphisms (dashed morphisms are explained below)
\[
\xymatrix{& \mathbb{W}_0(\Z)   \ar[rd] & & \\
& & \mathbb{M}(\Z) \ar[r] \ar[d] & \mathbb{W}(\Z) \ar[d]^{\mathghost} \\
\Z[\mathbb{L}] \ar@{.>}[ruu]^{\zeta_{\F_p}} \ar@{.>}[rru]_{\zeta_{\F_1}} & & \mathbb{L}(\Z) \ar[r] _i & \Z^{\infty}} \]
with the ghost-map $\mathghost = t \tfrac{d}{dt} log(-)$ and $\mathbb{M}(\Z)$ the pull-back of $\mathghost$ and the natural inclusion map $i$. One might speculate that the relevant counting measures $\mu : K_0(\mathbf{Var}_{\Z}) \rTo \Z$ are those which determine a ring-morphism $\zeta_{\mu} : K_0(\mathbf{Var}_{\Z}) \rTo \mathbb{M}(\Z)$, with those factoring over $\mathbb{W}_0(\Z)$ corresponding to the non-archimedean factors, and the remaining ones related to the $\Gamma$-factors.

\vskip 3mm

This is motivated by our description of the $\F_1$-zeta function of Lieber, Manin and Marcolli in \cite{LieberManinMarcolli}. Here, one considers integral schemes with a decomposition into tori $\mathbb{G}_m^n$ as $\F_1$-varieties and with morphisms respecting the decomposition and with all restrictions to tori being morphisms of group schemes. The corresponding Grothendieck ring $K_0(\mathbf{Var}_{\F_1}^{tor})$ can then be identified with the subring $\Z[\mathbb{L}]$ of $K_0(\mathbf{Var}_{\C})$. Kapranov's motivic zeta function induces a natural $\lambda$-ring structure on $\Z[\mathbb{L}]$ and we can also define a bi-ring structure on it by taking $\mathbb{D}=\mathbb{L}-2$ to be a primitive generator. By right adjointness we then have natural one-to-one correspondences
\[
\mathbf{comm}^+_{bi}(\Z[\mathbb{L}],\mathbb{L}(\Z)) \leftrightarrow \mathbf{comm}(\Z[\mathbb{L}],\Z) \leftrightarrow \mathbf{comm}^+_{\lambda}(\Z[\mathbb{L}],\mathbb{W}(\Z)) \]
To a counting measure $\mathbb{L} \mapsto m$ corresponds a $\lambda$-ring morphism $\zeta_m : \Z[\mathbb{L}] \rTo \mathbb{W}(\Z)$ which factors through $\mathbb{W}_0(\Z)$ and coincides with $\zeta_{\F_p}$ when $m=p$. If $X$ is an integral scheme with toric decomposition, its $\F_1$-zeta function is defined to be the ringmorphism
\[
\zeta_{\F_1} : \Z[\mathbb{L}] \rTo \mathbb{W}(\Z) \qquad \zeta_{\F_1}(X,t) = exp(\sum_{r \geq 1} \# X(\F_{1^m}) \frac{t^r}{r}) \]
with $\# X(\F_{1^m})$ being the total number of $m$-th roots of unity in the tori making up $X$, see \cite{LieberManinMarcolli}. This $\zeta_{\F_1}$ is not a $\lambda$-ring morphism and does not factor through $\mathbb{W}_0(\Z)$. However, the counting measure $\mathbb{L} \mapsto 3$ corresponds to a bi-ring morphism $c_3 : \Z[\mathbb{L}] \rTo \mathbb{L}(\Z)$ which factors through $\mathbb{M}(\Z)$ and such that the composition with $\mathbb{M}(\Z) \rTo \mathbb{W}(\Z)$ is the zeta-morphism $\zeta_{\F_1}$.

\vskip 3mm

\subsection{Structure of this paper} 

In section~\ref{measures}
we use right adjointness of the functor $\mathbb{W}(-)$ to give quick proofs of the facts that  the pre $\lambda$-structure on $K_0(\mathbf{Var}_{\C})$ given by Kapranov's motivic zeta function does not define a $\lambda$-ring structure, and that its universal motivic measure is not exponentiable.

In section~\ref{BC} we relate the versions of $\F_1$-geometry determined by $\lambda$-rings resp. bi-rings to the concrete resp. abstract Bost-Connes systems associated to cyclotomic Bost-Connes data as in \cite{MarTab}. This allows to have relative versions of $\mathbb{W}_0(\Z)$ and $\mathbb{L}(\Z)$ by imposing conditions on the eigenvalues of actions of Frobenii on (co)homology or on the roots and poles of zeta-polynomials.

In section~\ref{F1} we study counting measures on the Grothendieck ring of torified integral schemes, proving the results mentioned above. It turns out that the pull-back $\mathbb{M}(\Z)$ of $\mathbb{W}(\Z)$ and $\mathbb{L}(\Z)$ might be the appropriate  receptacle for local factors of zeta functions of integral schemes. These results can be extended to other subrings of $K_0(\mathbf{Var}_{\Z})$ which are $\lambda$-rings and admit a bi-ring structure.

In section~\ref{canonical} we introduce the category of all linear dynamical systems which plays the same role for $\mathbb{L}(\Z)$ as does the endomorphism category for $\mathbb{W}_0(\Z)$. To completely reachable systems we associate their transfer functions which are strictly proper rational functions. As such, these systems may be relevant in the study of zeta-polynomials, as introduced by Manin in \cite{ManinZeta}.

\vskip 3mm

{\bf Acknowledgements} This paper owes much to recent work of Yuri I. Manin, Matilde Marcolli and co-authors, \cite{ManinMarcolli},\cite{LieberManinMarcolli} and \cite{MarTab}. Unconventional symbols are taken from the \LaTeX-package {\tt halloweenmath} \cite{Mezzetti}, befitting the current topic.

\section{Motivic measures on $K_0(\mathbf{Var}_k)$} \label{measures}

Let $\mathbf{Var}_k$ be the category of varieties over a field $k$. The Grothendieck ring $K_0(\mathbf{Var}_k)$ is the quotient of the free abelian group on isomorphism classes $[X]$ of varieties by the relations $[X]=[Y]+[X-Y]$ whenever $Y$ is a closed subvariety of $X$, and multiplication is induced by products of varieties, that is, $[X].[Y]=[X \times Y]$. As the structure of $K_0(\mathbf{Var}_k)$ is fairly mysterious, we try to probe its properties via motivic measures.

\begin{definition}
A {\em motivic measure on $K_0(\mathbf{Var}_k)$} with values in a commutative ring $R$ is a ringmorphism 
\[
\mu : K_0(\mathbf{Var}_k) \rTo R \]
\end{definition}

The archetypical example of a motivic measure on the Grothendieck ring of varieties over a finite field $\F_q$ is the {\em counting measure} with values in $\Z$
\[
\mu_{\F_q} : K_0(\mathbf{Var}_{\F_q}) \rTo \Z \qquad [X] \mapsto \# X(\F_q) \]
An example of a motivic measure on the Grothendieck ring of complex varieties $K_0(\mathbf{Var}_{\C})$ with values in $\Z$ is the {\em Euler characteristic measure}
\[
\chi_c : K_0(\mathbf{Var}_{\C}) \rTo \Z \qquad [X] \mapsto \chi_c(X) = \sum_i (-1)^i dim_{\Q}~H^i_c(X^{an},\Q) \]

\vskip 3mm

There are plenty of motivic measures with values in other rings such as the {\em Hodge characteristic measure} $\mu_H$ with values in $\Z[u,v]$, see \cite[\S 4.1]{LiuSebag},  the {\em Poincar\'e characteristic measure} $P_X$ with values in $\Z[u]$, see \cite[\S 4.1]{LiuSebag}, the {\em Gillet-Soul\'e measure} $\mu_{GS}$ with values in the Grothendieck ring if Chow motives, see \cite{GilletSoule}.

\vskip 3mm

Of particular importance to us are the 'exotic' Larsen-Luntz measure $\mu_{LL}$ on $K_0(\mathbf{Var}_{\C})$ with values in the quotient field of the monoid ring $\Z[C]$ with $C$ the multiplicative monoid of polynomials in $\Z[t]$ with positive leading coefficient, see \cite{LarsenLuntz}, and  the {\em universal motivic measure}, which is the identity morphism $id : K_0(\mathbf{Var}_k) \rTo K_0(\mathbf{Var}_k)$.

\vskip 3mm

For a commutative ring $R$ let $\mathbb{W}(R)$ be the set $1 + tR[[t]]$ of all formal power series over $R$ with constant term equal to one, and let multiplication of formal power series be the {\em addition} on $\mathbb{W}(R)$. We say that $R$ admits a {\em pre $\lambda$-structure} if there exists a morphism of additive groups
\[
\lambda_t : R \rTo \mathbb{W}(R)=1+tR[[t]] \qquad a \mapsto \lambda_t(a) = 1 + at + \hdots = \sum_{m \geq 0} \lambda^m(a) t^m \]
that is, it satisfies $\lambda_0(a)=1, \lambda_1(a)=a$, and
\[
\lambda_t(a+b) = \lambda_t(a).\lambda_t(b) \quad \text{that is} \quad \lambda^m(a+b) = \sum_{i+j=m} \lambda^i(a)\lambda^j(b) \]
Given a pre $\lambda$-structure $\lambda_t$ on $R$ we can define the {\em Adams operations} $\Psi_m$ on $R$ via
\[
t \frac{d}{dt}~log(\lambda_t(a)) = t \frac{1}{\lambda_t(a)} \frac{d \lambda_t(a)}{dt} = \sum_{m \geq 1} \Psi_m(a) t^m \]
and note that for all $m \in \N$ and all $a,b \in R$ we have  $\Psi_m(a+b)=\Psi_m(a) + \Psi_m(b)$. We say that a pre $\lambda$-ring $R$ is a {\em $\lambda$-ring} if for all $m,n \in \N$ we have these conditions on the Adams operations
\[
\Psi_m(a.b) = \Psi_m(a).\Psi_m(b) \quad \text{and} \quad \Psi_m \circ \Psi_n = \Psi_n \circ \Psi_m \]
Equivalently, if we define a multiplication $\ast$ on $\mathbb{W}(R)$ induced by the functorial requirement that $(1-at)^{-1} \ast (1-bt)^{-1} = (1-abt)^{-1}$ for all $a,b \in R$, then the map $\lambda_t$ is a morphism of rings. For more on $\lambda$-rings, see \cite{Hazewinkel}, \cite{Knutson} and \cite{Wilkerson}.

\vskip 3mm

A morphism $\phi : (R,\lambda_t) \rTo (R',\lambda'_t)$ between two $\lambda$-rings is a ringmorphism such that for all $a \in R$ we have that $\lambda'_t(\phi(a)) = \mathbb{W}(\phi)(\lambda_t(a))$ where $\mathbb{W}(\phi)$ is the map on $\mathbb{W}(R)=1+tR[[t]]$ induced by $\phi$. With $\mathbf{comm}^+_{\lambda}$ we will denote the category of all (commutative) $\lambda$-rings. If $\mathbf{comm}$ is the category of all commutative rings, then
\[
\mathbb{W}~:~\mathbf{comm} \rTo \mathbf{comm}^+_{\lambda} \qquad A \mapsto \mathbb{W}(A) \]
is a functor, which is right adjoint to the forgetful functor $F : \mathbf{comm}^+_{\lambda} \rTo \mathbf{comm}$. That is, for every $\lambda$-ring $(R,\lambda_t)$ and every commutative ring $A$ we have a natural one-to-one corespondence
\[
\mathbf{comm}^+_{\lambda}(R,\mathbb{W}(A)) \leftrightarrow \mathbf{comm}(R,A) \qquad \phi \leftrightarrow \mathghost_1 \circ \phi \]
with the {\em ghost components} $\mathghost_m : \mathbb{W}(A) \rTo A$ defined by
\[
t \frac{1}{P} \frac{d P }{dt} = \sum_{m=1}^{\infty} \mathghost_m(P) t^m \qquad \text{for all $P \in \mathbb{W}(A)=1+tA[[t]] $} \]

\vskip 3mm

{\em Kapranov's motivic zeta function} $\zeta$ defines a natural pre $\lambda$-structure on $K_0(\mathbf{Var}_k)$
\[
\zeta : K_0(\mathbf{Var}_k) \rTo \mathbb{W}(K_0(\mathbf{Var}_k)) \quad [X] \mapsto \zeta_X(t) = 1 + [X]t+[S^2X]t^2 + [S^3X]t^3 + \hdots \]
where $S^nX = X^n/S_n$ is the $n$-th symmetric product of $X$.

\begin{definition} A motivic measure $\mu : K_0(\mathbf{Var}_k) \rTo R$ with values in $R$ is said to be {\em exponentiable} if the uniquely determined map $\zeta_{\mu} : K_0(\mathbf{Var}_k) \rTo \mathbb{W}(R)$ by
\[
\zeta_{\mu}([X]) = 1 + \mu([X])t + \mu([S^2 X])t^2 + \mu([S^3 X])t^3 + \hdots \]
is a ringmorphism.
\end{definition}

Again, the archetypical example being the counting measure $\mu_{\F_q}$ on $K_0(\mathbf{Var}_{\F_q})$ which is exponentiable, with corresponding zeta-function
\[
\zeta_{\mu_{\F_q}} : K_0(\mathbf{Var}_{\F_q}) \rTo \mathbb{W}(\Z) \qquad \zeta_{\mu_{\F_q}}([X]) = \sum_{m=1}^{\infty} \# X(\F_{q^m}) t^m = Z_{\F_q}(X,t) \]
the classical Hasse-Weil zeta function, see \cite[Prop. 8]{Naumann} or \cite[Thm. 2.1]{Ramachandran}. Also the Euler characteristic measure on $K_0(\mathbf{Var}_{\C})$ is exponentiable with corresponding zeta function
\[
\zeta_{\mu_c} : K_0(\mathbf{Var}_{\C}) \rTo \mathbb{W}(\Z) \qquad \zeta_{\mu_c}([X]) = \frac{1}{(1-t)^{\chi_c(X)}} \]
However, as shown in \cite[\S 4]{RamaTabu} the Larsen-Luntz motivic measure $\mu_{LL}$ on $K_0(\mathbf{Var}_{\C})$ is {\em not} exponentiable. For this would imply that
\[
\zeta_{\mu_{LL}}(C_1 \times C_2) = \zeta_{\mu_{LL}}(C_1) \ast \zeta_{\mu_{LL}}(C_2) \]
for any pair of projective curves $C_1$ and $C_2$. Kapranov proved that $\zeta_{\mu}(C)$ is a rational function for every curve and every motivic measure, which would imply that $\mu_{LL}(C_1 \times C_2)$ would be rational too, by \cite[Prop. 4.3]{RamaTabu}, which contradicts \cite[Thm 7.6]{LarsenLuntz} in case $C_1$ and $C_2$ have genus $\geq 1$.

\vskip 3mm

 It is a natural to ask whether the pre $\lambda$-structure on $K_0(\mathbf{Var}_k)$ defined by Kapranov's motivic zeta function defines a $\lambda$-ring structure on $K_0(\mathbf{Var}_k)$, see \cite[\S 3 Questions]{Ramachandran} or \cite[\S 2.2]{Gorsky}. The following is well-known to the experts, but we cannot resist including the short proof.

\begin{proposition} If Kapranov's motivic zeta function makes $K_0(\mathbf{Var}_k)$ into a $\lambda$-ring, then every motivic measure 
\[
\mu : K_0(\mathbf{Var}_k) \rTo R \]
 is exponentiable. 
 
 As a consequence, Kapranov's zeta function does not define a $\lambda$-ring structure on  $K_0(\mathbf{Var}_{\C})$.
\end{proposition}

\begin{proof} If $K_0(\mathbf{Var}_k)$ is a $\lambda$-ring, then by right adjunction of $\mathbb{W}(-)$ with respect to the  forgetful functor, we have a natural one-to-one correspondence
\[
\mathbf{comm}(K_0(\mathbf{Var}_k),R) \leftrightarrow \mathbf{comm^+_{\lambda}}(K_0(\mathbf{Var}_k),\mathbb{W}(R)) \]
and under this correspondence the motivic measure $\mu$ maps to a unique  $\lambda$-ring morphism $\zeta_{\mu} : K_0(\mathbf{Var}_k) \rTo \mathbb{W}(R)$. 

Because  the  Larsen-Luntz motivic measure $\mu_{LL}$ on $K_0(\mathbf{Var}_{\C})$ is not exponentiable, it follows that $K_0(\mathbf{Var}_{\C})$ cannot be a $\lambda$-ring.
\end{proof}

Another immediate consequence is this negative answer to \cite[\S 3 Questions]{Ramachandran}.

\begin{proposition} The universal motivic measure on $K_0(\mathbf{Var}_{\C})$ is not exponentiable.
\end{proposition}

\begin{proof} By functoriality, any motivic measure $\mu : K_0(\mathbf{Var}_{\C}) \rTo R$ gives rise to a morphism of $\lambda$-rings $\mathbb{W}(\mu) : \mathbb{W}(K_0(\mathbf{Var}_{\C})) \rTo \mathbb{W}(R)$. 

If the universal measure would be exponentiable, this would give a ringmorphism $\zeta : K_0(\mathbf{Var}_{\C}) \rTo \mathbb{W}(K_0(\mathbf{Var}_{\C}))$ and composition
\[
\mathbb{W}(\mu) \circ \zeta : K_0(\mathbf{Var}_{\C}) \rTo \mathbb{W}(R) \]
would then imply that $\mu$ is exponentiable, which cannot happen for $\mu_{LL}$.
\end{proof}

An important condition on a motivic measure $\mu : K_0(\mathbf{Var}_k) \rTo R$ is its {\em rationality}. In order to define this, we need to recall the {\em endomorphism category} and its Grothendieck ring, see \cite{Almkvist} and \cite{Grayson}.

\vskip 3mm

For a commutative ring $R$ consider the category $\mathcal{E}_R$ consisting of pairs $(E,f)$ where $E$ is a projective $R$-module of finite rank and $f$ is an endomorphism of $E$. Morphisms in $\mathcal{E}_R$ are module morphisms compatible with the endomorphisms. There is a duality $(E,f) \leftrightarrow (E^*,f^*)$ on $\mathcal{E}_R$ and we have $\oplus$ and $\otimes$ operations
\[
(E_1,f_1) \oplus (E_2,f_2) = (E_1 \oplus E_2,f_1 \oplus f_2) \quad (E_1,f_1) \otimes (E_2,f_2) = (E_1 \otimes E_2,f_1 \otimes f_2) \]
with a zero object $\mathbf{0} = (0,0)$ and a unit object $\mathbf{1}=(R,1)$. These operations
turn the Grothendieck ring $K_0(\mathcal{E}_R)$ into a commutative ring, having an ideal consisting of the pairs $(E,0)$, with quotient ring $\mathbb{W}_0(R)$.

The ring $\mathbb{W}_0(R)$ comes equipped with Frobenius ring endomorphisms $Fr_n(E,f)=(E,f^n)$, Verschiebung additive maps 
\[
V_n(E,f) = (E^{\oplus n},\begin{bmatrix} 0 & 0 & 0 & \hdots & 0 & f \\ 1 & 0 & 0 & \hdots & 0 & 0 \\ 0 & 1 & 0 & \hdots & 0 & 0 \\ \vdots & & \ddots & & & \vdots \\  \vdots & & &  \ddots & & \vdots \\ 0 & 0 & 0 & \hdots & 1 & 0 \end{bmatrix}) \]
and ghost ringmorphisms $\mathghost_n(E,f) = Tr(f^n) : \mathbb{W}_0(R) \rTo R$. For various relations among the maps $Fr_n,V_n$ and $\mathghost_n$ see for example \cite[Prop. 2.2]{ConnesConsaniBC}.

\vskip 3mm

The connection between Almkvist's functor $\mathbb{W}_0(-)$ and $\mathbb{W}(-)$ is given by the ringmorphisms
\[
L_R~:~\mathbb{W}_0(R) \rTo \mathbb{W}(R)  \qquad L_R(E,f) = \frac{1}{det(1 - t M_f)} \]
where $M_f$ is the matrix associated to $f$ (that is, if $f = \sum_i x_i^* \otimes x_i \in End_R(E) = E^* \otimes E$, then $M_f = (a_{ij})_{i,j}$ with $a_{ij} = x_i^*(x_j)$. By \cite[Thm 6.4]{Almkvist} we know that $L_R$ is injective with image all {\em rational} formal power series of the form
\[
\frac{1+a_1 t + \hdots + a_n t^n}{1 + b_1 t + \hdots + b_m t^m} \qquad a_i,b_i \in R, m,n \in \N_+  \]

\begin{definition}
We say that a motivic measure $\mu : K_0(\mathbf{Var}_k) \rTo R$ is {\em rational} if it is exponentiable and if the corresponding zeta-function $\zeta_{\mu}$ factors through $\mathbb{W}_0(RT)$. That is, there is a unique ringmorphism
\[
r_{\mu} : K_0(\mathbf{Var}_k) \rTo \mathbb{W}_0(R) \]
such that $\zeta_{\mu} = L_R \circ r_{\mu}$.
\end{definition}

By a classic result of Dwork we know that the counting measure $\mu_{\F_q}$ is rational, as is the Euler characteristic measure $\mu_c$.

\section{Cyclotomic Bost-Connes data} \label{BC}

Let $R$ be an integral domain with field of fractions $K$ of characteristic zero and with algebraic closure $\overline{K}$. Let $\overline{K}^*_{\times}$ be the multiplicative group of all non-zero elements and $\pmb{\mu}_{\infty}$ the subgroup consisting of all roots of unity. The power maps $\sigma_n : x \mapsto x^n$ for $n \in \N_+$ form a commuting family of endomorphisms of $\overline{K}^*_{\times}$ and its subgroups. Following M. Marcolli en G. Tabuada in \cite{MarTab} we define:

\begin{definition} A {\em cyclotomic Bost-Connes datum} is a divisible subgroup $\Sigma$
\[
\pmb{\mu}_{\infty} \subseteq \Sigma \subseteq \overline{K}^*_{\times} \]
stable under the action of the Galois group $G=Gal(\overline{K}/K)$.
\end{definition}

The subgroup $\Sigma$ should be considered as 'generalised' Weil numbers (recall that for each prime power $q=p^r$ the Weil $q$-numbers are an instance, see \cite[Example 4]{MarTab}). 

Observe that cyclotomic Bost-Connes data are special cases of {\em concrete} Bost-Connes data as in \cite[Def. 2.3]{MarTab} with the endomorphisms $\sigma_n$ the $n$-th power maps $\sigma_n(x)=x^n$ and $\rho_n(x) = \pmb{\mu}_n \sqrt[n]{x} \subset \Sigma$. In \cite[\S 4]{MarTab} Marcolli and Tabuada associate to a cyclotomic Bost-Connes system with $\overline{K}=\overline{\Q}$ a quantum statistical mechanical system.
Further, in \cite[\S 2]{MarTab} both {\em concrete} and {\em abstract} Bost-Connes systems are associated to a cyclotomic Bost-Connes datum $\Sigma$. We will relate these  to  $\F_1$-geometries. 

\vskip 3mm

A powerful idea, due to Jim Borger \cite{Borger1} and \cite{Borger2}, to construct 'geometries' under $\mathbf{Spec}(\Z)$ is to consider a subcategory $\mathbf{comm^+_X}$ of commutative rings $\mathbf{comm}$ which allows a right adjoint $R$ to the forgetful functor $F : \mathbf{comm^+_X} \rTo \mathbf{comm}$. 

The additional structure $\mathbf{X}$ should be thought of as descent data from $\Z$ to $\F_1$, the elusive field with one element. As a consequence, the commutative ring $F(R(\Z))$ can then be considered to be the coordinate ring of the arithmetic square $\mathbf{Spec}(\Z) \times_{\mathbf{Spec}(\F_1)} \mathbf{Spec}(\Z)$. 

We propose to view the object $R(\Z) \in \mathbf{comm^+_X}$ as a  receptacle for motivic data. That is, (co)homology groups with actions of Frobenii and zeta-functions  determine elements in $R(\Z)$ and the subobject in $\mathbf{comm^+_X}$ they generate can then be seen as its representative  in the corresponding version of $\F_1$-geometry.

\subsection{Concrete Bost-Connes systems and $\mathbf{comm^+_{\lambda}}$}

 Following \cite[Def. 2.6]{MarTab} one associates to $\Sigma$ the {\em concrete Bost-Connes system} which consists of the integral group ring $\Z[\Sigma]$ equipped with
\begin{enumerate}
\item{the induced $G=Gal(\overline{K}/K)$-action,}
\item{$G$-equivariant ring endomorphisms $\tilde{\sigma}_n$ induced by $\tilde{\sigma}_n(x)=x^n$ for all $x \in \Sigma$,}
\item{$G$-equivariant $\Z$-module maps $\tilde{\rho}_n$ induced by $\tilde{\rho}_n(x) = \sum_{x' \in \rho_n(x)} x'$ for all $x \in \Sigma$.}
\end{enumerate}

\begin{proposition} For a cyclotomic Bost-Connes datum $\Sigma$, the concrete Bost-Connes system $(\Z[\Sigma],\tilde{\sigma}_n,\tilde{\rho}_n)$ is a sub-system of $(\mathbb{W}_0(\overline{K}),Fr_n,V_n)$.
\end{proposition}

\begin{proof}
From \cite[Prop. 2.3]{ConnesConsaniBC} we recall that $\mathbb{W}_0(\overline{K})$ is isomorphic to the integral group ring $\Z[\overline{K}^*_{\times}]$ via the map which assigns to $(E,f)$ the divisor of non-zero eigenvalues of $f$ (with multiplicities). 

Under this isomorphism the Frobenius maps $Fr_n$ becomes $\tilde{\sigma}_n$ and the Verschiebung $V_n$ the map $\tilde{\rho}_n$ for the cyclotomic Bost-Connes datum $\overline{K}^*_{\times}$.
\end{proof}

\vskip 3mm

\begin{definition} For a cyclotomic Bost-Connes datum $\Sigma$, let
$\mathcal{E}_{\Sigma,R} $ be the full subcategory of $\mathcal{E}_R$ consisting of pairs $(E,f)$ with $E$ a projective $R$-module and $M_f$ a $\overline{K}$-diagonalisable matrix having all its eigenvalues in $\Sigma$.  With $\mathbb{W}_0(\Sigma,R)$ we denote the subring of $\mathbb{W}_0(R)$ generated by $\mathcal{E} _{\Sigma,R}$.\end{definition}

\begin{example} \label{MorseSmale}
Consider Yuri I. Manin's idea to replace the action of the Frobenius map on \'etale cohomology of an $\F_q$-variety at $q=1$ by pairs $(H_k(M,\Z),f_{\ast k})$ where $f_{\ast k}$ is the action of a Morse-Smale diffeomorphism $f$ on a compact manifold $M$ upon its homology $H_k(M,\Z)$,  \cite[\S 0.2]{ManinCyclotomy}. This implies that each $f_{\ast k}$ is quasi-unipotent, that is all its eigenvalues are roots of unity. This fits in with Manin's view that $1$-Frobenius morphisms acting upon their (co)homology have eigenvalues which are roots of unity.

In \cite[\S 2]{ManinMarcolli}, Manin and Matilde Marcolli assign an object in $\mathbf{comm^+_{\lambda}}$  to the Morse-Smale setting $(M,f)$ as follows. Each $H_k(M,\Z)$ is viewed as a $\Z[t,t^{-1}]$-module by letting $t$ act as $f_{\ast k}$. Next, they consider the minimal category $\mathcal{C}_M$ of $\Z[t,t^{-1}]$-modules, containing all $H_k(M,\Z)$, and closed with respect to direct sums, tensor products and exterior products. Then, its Grothendieck ring $K_0(\mathcal{C}_M)$ comes equipped with a $\lambda$-ring structure coming from the exterior products, which is
 then said to be the representative of $\{ (H_k(M,\Z),f_{\ast k});k \}$ in $\F_1$-geometry, see \cite[Def. 2.4.2]{ManinMarcolli}.

 Alternatively, one can assign to each $(H_k(M,\Z),f_{\ast k})$ the element 
\[
det(1-t (f_{\ast k} | H_k(M,\Z)))^{-1} \in 1 + t \Z[[t]] =  \mathbb{W}(\Z) \]
and consider the $\lambda$-subring of $\mathbb{W}(\Z)$ generated by these elements. Clearly, all $(H_k(M,\Z),f_{\ast k})$ lie in $\mathcal{E} _{\pmb{\mu}_{\infty},\Z}$.
\end{example}

\subsection{Abstract Bost-Connes systems and $\mathbf{comm^+_{bi}}$}

Following \cite[Def. 2.5]{MarTab} one can associate to a cyclotomic Bost-Connes datum $\Sigma$ the {\em abstract Bost-Connes system} which consists of the Galois-invariants of the group ring of $\Sigma$ over $\overline{K}$, that is,
\begin{enumerate}
\item{the $K$-algebra $\overline{K}[\Sigma]^{Gal(\overline{K}/K)}$, equipped with}
\item{$K$-algebra morphisms $\tilde{\sigma}_n$ induced by $x \mapsto x^n$ for all $x \in \Sigma$, and}
\item{$K$-linear maps $\tilde{\rho}_n$ induced by $x \mapsto_{x' \in \rho_n(x)} x'$ for all $x \in \Sigma$.}
\end{enumerate}
Clearly, $\overline{K}[\Sigma]^G$ is a Hopf-algebra and from \cite[Thm. 1.5.(iv)]{MarTab} we recall that the affine group $K$-scheme $\mathbf{Spec}(\overline{K}[\Sigma]^G)$ agrees with the Galois group of the neutral Tannakian category $\mathbf{Aut}^{\overline{K}}_{\Sigma}(\Q)$ consisting of pairs $(V,\Phi)$ with $V$ a finite dimensional $K$-vectorspace and
\[
\Phi : V \otimes \overline{K} \rTo V \otimes \overline{K} \]
a $G$-equivariant diagonalisable automorpism all of whose eigenvalues belong to $\Sigma$, that is, the category $\mathcal{E}_{\Sigma,K} $ introduced above.

\vskip 3mm

In \cite{LBrecursive} we proposed to consider the category $\mathbf{comm^+_{bi}}$ of all (torsion free) commutative and co-commutative $\Z$-birings. This time, the forgetful functor $F : \mathbf{comm^+_{bi}} \rTo \mathbf{comm}$ has as right adjoint $C(-)$ where $C(A)$ is the free co-commutative co-ring on $A$. In particular, $C(\Z)=\mathbb{L}(\Z)$, the coring of all integral linear recursive sequences, equipped with the Hadamard product, see \cite[Thm. 2]{LBrecursive}.

\vskip 3mm

For a commutative domain $R$, consider the polynomial ring $R[t]$ with coring structure defined by letting $t$ be a group-like element, that is, $\Delta(t)=t \otimes t$ and $\epsilon(t)=1$. 

The full linear dual $R[t]^{\ast}$ can be identified with the module of all infinite sequences $f=(f_n)_{n=0}^{\infty} \in R^{\infty}$ with $f(t^n) = f_n$. $\mathbb{L}(R)$ will be $R[t]^o$, that is, the submodule of all sequences $f$ such that $Ker(f) = (m(t))$ with $m(t) = t^r - a_1 t^{r-1} - \hdots - a_r$ is a monic polynomial. As $f(t^nm(t))=0$ it follows that $f$ is a linear recursive sequence, that is, for all $n \geq r$ we have 
$f_n = a_1 f_{n-1} + a_2 f_{n-2} + \hdots + a_r f_{n-r}$. Therefore,
\[
\mathbb{L}(R) =  R[t]^o = \underset{\rightarrow}{\text{lim}}~(\frac{R[t]}{(m(t))})^* \]
where the limit is taken over the multiplicative system of monic polynomials with coefficients in $R$. 

We define a coring structure on $\mathbb{L}(R)$ dual to the ring structure on $R[t]/(m(t))$. With this coring structure, $\mathbb{L}(R)$ becomes an integral biring if we equip $\mathbb{L}(R)$ with the Hadamard product of sequences, that is, componentwise multiplication $(f.g)_n = f_n.g_n$ and unit $1=(1,1,1,\hdots)$. 

\vskip 3mm

If $K$ is a field of characteric zero, one can describe the co-algebra structure on $\mathbb{L}(K)$  explicitly, see  \cite{PetersonTaft} for more details. 

On the linear recursive sequence $f = (f_i)_{i=0}^{\infty} \in K^{\infty}$ the counit acts as $\epsilon(f) = f_0$, projection on the first component. To define the co-multiplication recall that the {\em Hankel matrix} $M(f)$ of the sequence $f$ is the symmetric $k \times k$ matrix
\[
H(f) = \begin{bmatrix} f_0 & f_1 & f_2 & \hdots & f_{k-1} \\
f_1 & f_2 & f_3 & \hdots & f_k \\
f_2 & f_3 & f_4 &  \hdots & f_{k+1} \\
\vdots & \vdots & \vdots & & \vdots \\
f_{k-1} & f_k & f_{k+1} & \hdots & f_{2k-2} \end{bmatrix} \]
with $k$ maximal such that $H(f)$ is invertible. If $H(f)^{-1} = (s_{ij})_{i,j} \in M_n(K)$ then we have in $\mathbb{L}(K)$
\[
\Delta(f) = \sum_{i,j=0}^{k-1} s_{ij} (D^i f) \otimes (D^j f) \]
where $D$ is the shift operator $D(f_0,f_1,f_2,\hdots) = (f_1,f_2,\hdots)$.  Clearly, if $K$ is the fraction field of $R$, and if a sequence $f \in \mathbb{L}(R)$ has Hankel matrix $H(f)$ with determinant a unit in $R$, the same formula applies for $\Delta(f)$ as $\mathbb{L}(R)$ is a sub-biring of $\mathbb{L}(K)$. In general however, $\Delta(f)$ cannot be diagonalized in terms of $f,Df,D^2f, \hdots$ with $R$-coefficients and we have no other option to describe the comultiplication than as the direct limit of linear duals of the ringstructures on $R[t]/(m(t))$.

\vskip 3mm

\begin{proposition}
For a cyclotomic Bost-Connes datum $\Sigma$, the Hopf-algebra $\overline{K}[\Sigma]^G$ describing the abstract Bost-Connes system  is a sub-bialgebra of $\mathbb{L}(K)$.
\end{proposition}

\begin{proof}
We can describe the bialgebra $\mathbb{L}(\overline{K})$ of linear recursive sequences over $\overline{K}$ using the structural results for commutative and co-commutative Hopf algebras over an algebraically closed field of characteristic zero, see \cite{LarsonTaft}. 

Let $T$ be the set of all sequences over $\overline{K}$ which are zero almost everywhere, then $T$ is a bialgebra ideal in $\mathbb{L}(\overline{K})$ and we have a decomposition
\[
\mathbb{L}(\overline{K}) = \overline{K}[t]^o \simeq \overline{K}[t,t^{-1}]^o \oplus T \]
One verifies that in the Hopf-dual $\overline{K}[t,t^{-1}]^o$ the group of group-like elements is isomorphic to the multiplicative group $\overline{K}^*_{\times}$, with $s \in \overline{K}^*_{\times}$ corresponding to the geometric sequence $(1,s,s^2,s^3,\hdots)$. Further, there is a unique primitive element corresponding to the sequence $d=(0,1,2,3,\hdots)$. Then, the structural result implies that, as bialgebras, we have an isomorphism
\[
\mathbb{L}(\overline{K}) \simeq (\overline{K}[\overline{K}^*_{\times}] \otimes \overline{K}[d]) \oplus T \]
As the Galois group $G=Gal(\overline{K}/K)$ acts on this bialgebra and as $\mathbb{L}(K) = \mathbb{L}(\overline{K})^G$, the claim follows.
\end{proof}

\begin{example}
Continuing Example~\ref{MorseSmale} on Morse-Smale diffeomorphism, as anticipated in \cite[remark 2.4.3]{ManinMarcolli}, in the $\mathbf{comm^+_{bi}}$-proposal, one can associate to each $(H_k(M,\Z),f_{\ast k})$ the element
\[
(Tr(f_{\ast k} | H_k(M,\Z)),Tr(f_{\ast k} ^2| H_k(M,\Z)),Tr(f_{\ast k}^3 | H_k(M,\Z)),\hdots ) \in \mathbb{L}(\Z) \]
and considers the sub-biring of $\mathbb{L}(\Z)$ generated by these elements.
\end{example}

\subsection{Motivic measures and $\mathbb{L}(R)$} 

By taking the trace of the Cayley-Hamilton polynomial we have a ghost ringmorphism
$\mathghost : \mathbb{W}_0(R) \rTo \mathbb{L}(R)$
\[
 (E,f) \mapsto (\mathghost_1(E,f),\mathghost_2(E,f),\hdots ) = (Tr(M_f),Tr(M_f^2),\hdots ) \]
 Further, we have a traditional ghost morphism $\mathghost : \mathbb{W}(R) \rTo R^{\infty}$ determined by $t \tfrac{d}{dt} log(-)$ on $\mathbb{W}(R) = 1 + tR[[t]]$
 \[
 \mathghost(f(t)) = (a_1,a_2,\hdots ) \quad \text{where} \quad t \frac{d}{dt} log(f(t)) = \sum_{m=1}^{\infty} a_m t^m \]
 
\begin{proposition} Let $R$ be a commutative ring and $\mu : K_0(\mathbf{Var}_k) \rTo R$ a motivic measure.  The measure $\mu$ is exponentiable if there exists a ringmorphism $\zeta_{\mu}$, and is rational if there is a ringmorphism $r_{\mu}$, making the diagram below commute
\[
\xymatrix{ & K_0(\mathbf{Var}_k) \ar[r]^{\mu} \ar@{.>}[rd]^{\zeta_{\mu}} \ar@{.>}[d]_{r_{\mu}} & R  \\
\mathcal{E}_R \ar@{.>}[r] & \mathbb{W}_0(R) \ar[r]^{L_R} \ar[d]_{\mathghost} & \mathbb{W}(R) \ar[d]^{\mathghost} \\
\mathcal{S}^{cr}_R \ar@{.>>}[r]^{\mathbat} & \mathbb{L}(R) \ar[r]^i & R^{\infty}} \]
The left-most maps are additive and multiplicative from the endomorphism category, resp. the category of completely reachable systems, to be defined in \S~\ref{canonical}.
\end{proposition}

\begin{proof} This follows from the definitions above and the fact that $log(L_R(E,f))=\sum_{m \geq 1} Tr(M_f^m) \tfrac{t^m}{m}$.
\end{proof}

\begin{example}
As a consequence, an exponentiable motivic measure $\mu$ assigns to a $k$-variety $X$ the element $\zeta_{\mu}([X]) \in \mathbb{W}(R)$, and a rational motivic measure $\mu$ assigns to $X$ elements $\mathghost(r_{\mu}([X])) \in \mathbb{L}(R)$ and $L_R(r_{\mu}([X])) \in \mathbb{W}(R)$.
\end{example}

\section{Motivic measures on $K_0(\mathbf{Var}^{tor}_{\F_1})$} \label{F1}

In this section we consider yet another approach to $\F_1$-geometry based on the notion of {\em torifications} as introduced by Lorscheid and Lopez Pena in \cite{LopezLorscheid} and generalized by Manin and Marcolli in \cite{ManinMarcolli0}.

A {\em torification} of a complex algebraic variety, defined over $\Z$, is a decomposition into algebraic tori
\[
X = \sqcup_{i \in I} T_i \quad \text{with} \quad T_i \simeq \mathbb{G}_m^{d_i} \]
We consider here {\em strong morphisms} between torified varieties (see \cite[\S 5.1]{LieberManinMarcolli} for weaker notions), that is a morphism of varieties, defined over $\Z$,
\[
f : X = \sqcup_{i \in I} T_i \rTo Y = \sqcup_{j \in J} T_j' \]
together with a map $h : I \rTo J$ of the indexing sets such that the restriction of $f$ to any torus
\[
f_i = f |_{T_i} : T_i \rTo T'_{h(i)} \]
is a morphism of algebraic groups. With $K_0(\mathbf{Var}_{\F_1}^{tor})$ we denote the Grothendieck ring generated by the strong isomorphism classes $[X=\sqcup_i T_i]$ of torified varieties, modulo the {\em scissor relations}
\[
[X = \sqcup_i T_i] = [Y=\sqcup_j T'_j] + [X \backslash Y=\sqcup_k T"_k] \]
whenever the decomposition in tori in the torifications of $Y$ and $X \backslash Y$ is a union of tori of the torification of $X$.
This condition is very strong and implies that the class of any torified variety in $K_o(\mathbf{Var}_{\F_1}^{tor})$ is of the form
\[
[X = \sqcup_i T_i] = \sum_{n \geq 0} a_n \mathbb{T}^n \quad \text{with $a_n \in \N_+$ and $\mathbb{T} = [\mathbb{G}_m] = \mathbb{L}-1 \in K_0(\mathbf{Var}_{\C})$} \]
That is,
\[
K_0(\mathbf{Var}_{\F_q}^{tor}) = \Z[\mathbb{T}] = \Z[\mathbb{L}] \subset K_0(\mathbf{Var}_{\C}) \]
with $\mathbb{L}=[\mathbb{A}^1]$ the Lefschetz motive. Whereas Kapranov's motivic zeta function does not make $K_0(\mathbf{Var}_{\C})$ into a $\lambda$-ring, it does define a $\lambda$-structure on certain subrings, including $\Z[\mathbb{L}]$, see \cite[\S 2.2 Example]{Gorsky}, with $S^n(\mathbb{L}) = \mathbb{L}^n$

\begin{proposition} \label{rat} Any motivic measure $\mu : K_0(\mathbf{Var}_{\F_1}^{tor}) \rTo R$ with values in a commutative ring $R$ is exponentiable and rational.
\end{proposition}

\begin{proof} Because $K_0(\mathbf{Var}_{\F_1}^{tor}) = \Z[\mathbb{L}]$ is a $\lambda$-ring, we have by right adjointness of $\mathbb{W}(-)$ a natural one-to-one correspondence
\[
\mathbf{comm}(K_0(\mathbf{Var}_{\F_1}^{tor}),R) \leftrightarrow \mathbf{comm}^+_{\lambda}(K_0(\mathbf{Var}_{\F_1}^{tor}),\mathbb{W}(R)) \]
with $\mu$ corresponding to a unique $\lambda$-ring morphism
\[
\zeta_{\mu} : K_0(\mathbf{Var}_{\F_1}^{tor}) \rTo \mathbb{W}(R) = 1 + t R[[t]] \qquad \mathbb{L} \mapsto 1 + rt + r^2t^2+ \hdots = \frac{1}{1-rt} \]
with $r=\mu(\mathbb{L})$. That is, $\mu$ is exponentiable and rational as it factors through the ringmorphism $r_{\mu} : K_0(\mathbf{Var}_{\F_1}^{tor}) \rTo \mathbb{W}_0(R)$ defined by $\mathbb{L} \mapsto [R,r]$.
\end{proof}

If we equip $K_0(\mathbf{Var}_{\F_1}^{tor}) = \Z[\mathbb{L}]$ with the bi-ring structure induced by letting $\mathbb{L}$ be a group-like generator, that is $\Delta(\mathbb{L})= \mathbb{L} \otimes \mathbb{L}$ and $\epsilon(\mathbb{L})=1$, we have a bi-ring morphism $c_{\mu} : K_0(\mathbf{Var}_{\F_1}^{tor}) \rTo \mathbb{L}(R)$ defined by $\mathbb{L} \mapsto (1,r,r^2,\hdots)$ making the diagram below commutative.
\[
\xymatrix{K_0(\mathbf{Var}_{\F_1}^{tor}) \ar@/^3ex/[rrd]^{\zeta_{\mu}} \ar@/_4ex/[rdd]_{c_{\mu}} \ar[rd]|{r_{\mu}} & & \\
& \mathbb{W}_0(R) \ar[r]^{L_{\Z}} \ar[d]_{\mathghost} & \mathbb{W}(R) \ar[d]^{\mathghost} \\
& \mathbb{L}(R) \ar[r] & R^{\infty}} \]

For example, any motivic measure with values in $\Z$ is of the form
\[
\mu_m : K_0(\mathbf{Var}_{\F_1}^{tor}) = \Z[\mathbb{L}] \rTo \Z \qquad \mathbb{L} \mapsto m+1 \]
and if $m+1=p$ with $p$ a prime number, the corresponding zeta function $\zeta_{\mu_m}(X,t)$ coincides with the Hasse-Weil zeta function of the reduction mod $p$ of the torified variety $X$. The reason for choosing $m+1$ rather than $m$ will be explained in \ref{countingF1} below.

\vskip 3mm

Similarly, we can define {\em $\F_{1^m}$-varieties} to be torified varieties $X = \sqcup T_i$ with the natural action of the group of $m$-th roots of unity $\pmb{\mu}_m$ on each torus $T_i$. As a consequence we have
\[
K_0(\mathbf{Var}_{F_{1^m}}^{tor}) = \Z[\mathbb{T}] = \Z[\mathbb{L} ] \]
and the previous result holds also for $K_0(\mathbf{Var}_{\F_{1^m}}^{tor})$.

\vskip 3mm

\subsection{Counting $\F_{1^m}$-points} \label{countingF1}

The motivic measure $\mu_{2m}$ can be interpreted as a 'counting measure' associated to the $\F_1$-extension $\F_{1^m}$. 

Indeed, in  \cite[Lemma 5.6]{LieberManinMarcolli} Joshua Lieber, Yuri I. Manin and Matilde Marcolli define for a torified variety $X$ with Grothendieck class $[X]=\sum_{i=0}^N a_i \mathbb{T}^i \in K_0(\mathbf{Var}_{\F_1}^{tor})$ that
\[
\# X(\F_{1^m}) = \sum_{i=0}^N a_i m^i \]
That is, $\# X(\F_1)$ counts the number of tori in the torified variety $X$, and $\# X(\F_{1^m})$ counts the number of $m$-th roots of unity in the tori-decomposition of $X$. Therefore, $\mu_m = \mu_{\F_{1^m}}$.

\vskip 3mm

In analogy with this Hasse-Weil zeta function of varieties over $\F_q$, Lieber, Manin and Marcolli then define the {\em $\F_1$- zeta function} to be the ring morphism, by \cite[Prop. 6.2]{LieberManinMarcolli}
\[
\zeta_{\F_1}~:~K_0(\mathbf{Var}_{\F_1}^{tor} ) \rTo \mathbb{W}(\Z) \qquad [X]=\sum_{k=0}^N a_k \mathbb{T}^k \mapsto exp(\sum_{k=0}^N a_k Li_{1-k}(t) ) \]
where $Li_s(t)$ is the polylogarithm function, that is, $Li_{1-k}(t) = \sum_{l \geq 1} l^{k-1} t^l$. 
This gives us a motivic measure on  $K_0(\mathbf{Var}_{\F_{1}}^{tor})$ with values in $\mathbb{W}(\Z)$, but it does not correspond to any of the zeta-functions $\zeta_{\mu_k}$ corresponding to the motivic measure $\mu_k$. In particular, $\zeta_{\F_1}$ is {\em not} a morphism of $\lambda$-rings.

\vskip 3mm

Mutatis mutandis we can define similarly the $\F_{1^m}$-zeta function, for the field extension $\F_{1^m}$ of $\F_1$, to be the ring morphism
\[
\zeta_{\F_{1^m}}~:~K_0(\mathbf{Var}_{\F_{1^m}}^{tor} ) \rTo \mathbb{W}(\Z) \qquad [X]=\sum_{k=0}^N a_k \mathbb{T}^k \mapsto exp(\sum_{k=0}^N a_k m^k  Li_{1-k}(t) ) \]
and again, this zeta function does not come from any of the motivic measures $\mu_k$ on $K_0(\mathbf{Var}_{\F_{1^m}})$.

\vskip 3mm

However, we can define another bi-ring (actually, Hopf-ring) structure on $K_0(\mathbf{Var}_{\F_{1^m}}^{tor})= \Z[\mathbb{T}]$ induced by taking $\mathbb{D}=\mathbb{T}-m$ (observe that $\# \mathbb{D}(\F_{1^m}) = 0$) to be the primitive generator, that is,
\[
\Delta(\mathbb{D}) = \mathbb{D} \otimes 1 + 1 \otimes \mathbb{D} \quad \text{and} \quad \epsilon(\mathbb{D}) = 0 \]
We will call this the {\em Lie algebra structure} on $K_0(\mathbf{Var}_{\F_{1^m}}^{tor})$.

\begin{proposition} \label{nonrat} If we equip $K_0(\mathbf{Var}_{\F_{1^m}}^{tor})=\Z[\mathbb{T}]$ with the Lie-algebra structure, then under the natural one-to-one correspondence
\[
\mathbf{comm}(K_0(\mathbf{Var}_{\F_{1^m}}^{tor}),\Z) \leftrightarrow \mathbf{comm}^+_{bi}(K_0(\mathbf{Var}_{\F_{1^m}}^{tor}),\mathbb{L}(\Z)) \]
the motivic measure $\mu_{2m} : K_0(\mathbf{Var}_{\F_{1^m}}^{tor}) \rTo \Z$ corresponds to a unique bi-ring morphism $c_{\mu_{2m}} : K_0(\mathbf{Var}_{\F_{1^m}}^{tor}) \rTo \mathbb{L}(\Z)$, making the diagram below commutative
\[
\xymatrix{K_0(\mathbf{Var}_{\F_{1^m}}^{tor}) \ar@/^3ex/[rrd]^{\zeta_{\F_{1^m}}} \ar@/_4ex/[rdd]_{c_{\mu_{2m}}} \ar@{.>}[rd]|{\not\exists} & & \\
& \mathbb{W}_0(\Z) \ar[r]^{L_{\Z}} \ar[d]_{\mathghost} & \mathbb{W}(\Z) \ar[d]^{\mathghost} \\
& \mathbb{L}(\Z) \ar[r] & \Z^{\infty}} \]
\end{proposition}

\begin{proof} By definition we have that $\zeta_{\F_{1^m}}(\mathbb{T}^i) = exp(\sum_{k \geq 1} m^i k^{i-1}t^k)$, and therefore, because $\mathghost$ corresponds to $t \frac{d}{dt} log(-)$, we have that
\[
 \mathghost(\zeta_{\F_{1^m}}(\mathbb{T}^i)) = (m^i,m^i 2^i,m^i 3^i, \hdots) =  \mathghost(\zeta_{\F_{1^m}}(\mathbb{T}))^i  \]
To enforce commutativity with a ringmorphism $c_{\mu}$ we must have that
\[
 c_{\mu}(\mathbb{T}) = (m,2m,3m,\hdots )=m.d+m.1 \]
for the primitive element $d = (0,1,2,\hdots) \in \mathbb{L}(\Z)$, that is, $\Delta(d)=d \otimes 1 + 1 \otimes d$ and $\epsilon(d)=0$ and with $1=(1,1,1,\hdots) \in \mathbb{L}(\Z)$.

But then, for the Lie algebra structure on $K_0(\mathbf{Var}_{\F_{1^m}}^{tor})$ we have that $c_{\mu}(\mathbb{D})$ is the primitive element $m.d \in \mathbb{L}(\Z)$, and therefore $c_{\mu}$ is the unique bi-ring morphism $K_0(\mathbf{Var}_{\F_{1^m}}^{tor}) \rTo \mathbb{L}(\Z)$ corresponding to the motivic measure $\mu_{2m} : K_0(\mathbf{Var}_{\F_{1^m}}^{tor}) \rTo \Z$ because the second component of $c_{\mu}(\mathbb{T}) = 2m$.

Suppose there would be a ringmorphism $r : K_0(\mathbf{Var}_{\F_{1^m}}^{tor}) \rTo \mathbb{W}_0(\Z)$, then we must have that $\mathghost(r(\mathbb{T} - m) )= m.d \in \mathbb{L}(\Z)$. By functoriality we have a commuting square
\[
\xymatrix{\mathbb{W}_0(\Z) \ar[r] \ar[d] & \mathbb{W}_0(\overline{\Q}) = \Z[\overline{\Q}^{\ast}_{\times}] \ar[d] \\
\mathbb{L}(\Z) \ar[r] & \mathbb{L}(\overline{\Q}) = (\overline{\Q}[\overline{\Q}^{\ast}_{\times}] \otimes \overline{\Q}[ d ]) \oplus K} \]
and $d$ does not lie in the image of the rightmost map.
\end{proof}

Because $K_0(\mathbf{Var}_{\F_1}^{tor})$ is both a $\lambda$-ring (with $\Psi_k(\mathbb{L}^i)=\mathbb{L}^{ki}$) and a bi-ring (with the Lie algebra structure with primitive element $\mathbb{D}=\mathbb{T}-1$) we have natural one-to-one correspondences
\[
\mathbf{comm}^+_{bi}(K_0(\mathbf{Var}_{\F_1}^{tor}),\mathbb{L}(\Z)) \leftrightarrow \mathbf{comm}(K_0(\mathbf{Var}_{\F_1}^{tor}),\Z) \leftrightarrow \mathbf{comm}^+_{\lambda}(K_0(\mathbf{Var}_{\F_1}^{tor}),\mathbb{W}(\Z)) \]
Under the left correspondence, the motivic measure $\mu_m$ defined by $\mu_m(\mathbb{T})=m$ corresponds to the bi-ring morphism
\[
b_m : K_0(\mathbf{Var}_{\F_1}^{tor}) = \Z[\mathbb{D}] \rTo \mathbb{L}(\Z) \qquad \mathbb{D} \mapsto (m-1).d = (0, m-1,2(m-1),\hdots ) \]
as $b_m(\mathbb{T}) = (1,m,2m-1,\hdots )$ and the coresponding ring-morphism to $\Z$ is composing with projection on the second factor.

Under the right correspondence, the motivic measure $\mu_m$ corresponds to the $\lambda$-ring morphism 
$l_m : K_0(\mathbf{Var}_{\F_1}^{tor~}) = \Z[\mathbb{L}] \rTo \mathbb{W}(\Z)$ 
\[
 \mathbb{L}  \mapsto \frac{1}{1-(m+1)t} = 1 + (m+1) t + (m+1)^2 t^2 + \hdots \]
as $l_m(\mathbb{T}) = (1-t).b_l(\mathbb{L}) = 1 + m t + m(m+1) t^2 + \hdots$ and the corresponding ring morphism to $\Z$ is $\mathghost_1(l_m(\mathbb{T})) = m$. 

It follows from propositions~\ref{rat} and \ref{nonrat} that these morphisms factor through the pull-back $\mathbb{M}(\Z)$.
\[
\xymatrix{\mathbb{W}_0(\Z)  \ar@/^3ex/[rrd]^{L_{\Z}} \ar@/_4ex/[rdd]_{\mathghost} \ar[rd] & & \\
& \mathbb{M}(\Z) \ar[r] \ar[d] & \mathbb{W}(\Z) \ar[d]^{\mathghost} \\
& \mathbb{L}(\Z) \ar[r] & \Z^{\infty}} \]
Motivated by this, one might view $\mathbb{M}(\Z)$ as the correct receptacle for ringmorphisms $K_0(\mathbf{Var}_{\Z}) \rTo \mathbb{W}(\Z)$ determined by a counting measure $K_0(\mathbf{Var}_{\Z}) \rTo \Z$. Here, local factors corresponding to non-archimedean places can be distinguished from the $\Gamma$-factors by the fact that they factor through $\mathbb{W}_0(\Z)$.

\vskip 3mm

\section{Linear systems and zeta-polynomials} \label{canonical}

The original motivation for proposing bi-rings as $\F_1$-algebras was to give a potential explanation of Manin's interpretation of Deninger's $\Gamma$-factor $\prod_{n \geq 0} \tfrac{s-n}{2 \pi}$ at complex infinity as the zeta function of (the dual of) infinite dimensional projective space $\mathbb{P}^{\infty}_{\F_1}$, see \cite[4.3]{ManinNumbers} and \cite[Intro]{ManinZeta}. In \cite{LBrecursive} a noncommutative moduli space was constructed using linear dynamical systems having the required motive. This suggests the introduction of the category $\mathcal{S}_R$ of discrete $R$-linear dynamical systems, which plays a similar role for $\mathbb{L}(R)$ as does the endomorphism category $\mathcal{E}_R$ for $\mathbb{W}_0(R)$ and $\mathbb{W}(R)$.

\vskip 3mm

For $R$ a commutative ring  consider the category $\mathcal{S}_R$ with objects quadruples $(E,f,v,c)$ with $E$ a projective $R$-module of finite rank, $f \in End_R(E)$, $v \in E$ and $c \in E^*$ and with morphisms $R$-module morphisms $\phi : E \rTo E'$ such that $\phi \circ f = f' \circ \phi$, $\phi(v)=v'$ and $c=c' \circ \phi$. A quadruple $(E,f,v,c)$ can be seen as an $R$-representation of the quiver
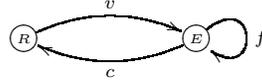
\[
\xymatrix{\vtx{R} \ar@/^2ex/[rr]^v & & \vtx{E} \ar@/^2ex/[ll]^c \ar@(ur,dr)^f} \]
and morphisms correspond to quiver-morphisms.

\vskip 3mm

Again, there is a duality $S=(E,f,v,c) \leftrightarrow S^*=(E^*,f^*,c^*,v^*)$ on $\mathcal{S}_R$ and we have $\oplus$ and $\otimes$ operations
\[
\begin{cases}
(E_1,f_1,v_1,c_1) \oplus (E_2,f_2,v_2,c_2) = (E_1 \oplus E_2,f_1 \oplus f_2, v_1 \oplus v_2,c_1 \oplus c_2) \\
(E_1,f_1,v_1,c_1) \otimes (E_2,f_2,v_2,c_2) = (E_1 \otimes E_2,f_1 \otimes f_2, v_1 \otimes v_2,c_1 \otimes c_2)
\end{cases}
\]
with a zero object $\mathbf{0}=(0,0,0,0)$ and a unit object $\mathbf{1}=(R,1,1,1)$.

\vskip 3mm

We will call a quadruple $S=(E,f,v,c)$ a {\em discrete $R$-linear dynamical system}. Borrowing terminology from system theory, see for example \cite[[VI.\S 5]{Tannenbaum}, we define:

\begin{definition} For $S=(E,f,v,c) \in \mathcal{S}_R$ with $E$ of rank $n$, we say that
\begin{enumerate}
\item{$S$ is {\em completely reachable} if $E$ is generated as $R$-module by the elements $\{ v,f(v),f^2(v),\hdots \}$.}
\item{$S$ is {\em completely observable} if the $R$-module morphism $\phi : E \rTo R^n$ given by $\phi(x) = (c(x),c(f(x)),\hdots,c(f^{n-1}(x)))$ is injective.}
\item{$S$ is a {\em canonical system} if $S$ is both completely reachable and completely observable.}
\item{$S$ is a {\em split system} if both $S$ and $S^*$ are completely reachable.}
\end{enumerate}
\end{definition}

\begin{definition}
There is  an additive and multiplicative bat-map
\[
\mathbat_R~:~\mathcal{S}_R \rTo \mathbb{L}(R) \qquad (E,f,v,c) \mapsto (c(v),c(f(v)),c(f^2(v)),c(f^3(v)),\hdots ) \]
sending a linear dynamical system to its input-output or transfer sequence. We say that a linear recursive sequence $s=(s_0,s_1,s_2,\hdots ) \in \mathbb{L}(R)$ is {\em realisable} by the system $(E,f,v,c) \in \mathcal{S}_R$ if $\mathbat_R(E,f,v,c)=s$.
\end{definition}

\begin{remark} In system theory, see for example  \cite[VI.\S 5]{Tannenbaum}, one relaxes the condition on the state-space $E$ which is merely an $R$-module and replaces the $rk(E)=n$ condition by the requirement that $E$ is generated by $n$ elements. 
\end{remark}

We will now prove that every element $s \in \mathbb{L}(R)$ is realisable by a completely reachable system and verify when this system is in addition canonical, respectively split.

For $s = (s_0,s_1,s_2,\hdots ) \in \mathbb{L}(R)$ satisfying the recurrence relation $s_n = a_1 s_{n-1} + a_2 s_{n-2} + \hdots + a_r s_{n-r}$ of depth $r$, valid for all $n \in \N$ with the $a_i \in R$. Consider the system $S_s=(E_s,f_s,v_s,c_s) \in \mathcal{S}_R$ with
\[
E_s = \frac{R[x]}{(x^r - a_1 x^{r-1} -  \hdots - a_r)}, \quad f_s=x. | E_s, \quad v_s=1 \in E_s, \quad c_s(x^i)=s_i \]
and consider the $r \times r$ matrix, with $r$ the depth of the recurrence relation
\[
H_r(s) = \begin{bmatrix} s_0 & s_1 & s_2 & \hdots & s_{r-1} \\
s_1 & s_2 & s_3 & \hdots & s_r \\
s_2 & s_3 & s_4 & \hdots & s_{r+1} \\
\vdots & \vdots & \vdots & & \vdots \\
s_{r-1} & s_r & s_{r+1} & \hdots & s_{2r-2} \end{bmatrix} \]

\begin{proposition} With notations as above, $s \in \mathbb{L}(R)$ is realisable by the system $S_s = (E_s,f_s,v_s,c_s) \in \mathcal{S}_R$, and
\begin{enumerate}
\item{$S_s$ is completely reachable, }
\item{$S_s$ is canonical if and only if $det(H_r(s)) \not= 0$,}
\item{$S_s$ is split if and only if $det(H_r(s)) \in R^{\ast}$.}
\end{enumerate}
\end{proposition}

\begin{proof}
Clearly, $E_s$ is a free $R$-module of rank $r$ and one verifies that $\mathbat_R(S_s) = s$. Further, $\{ v_s,f_s(v_s),f_s^2(v_s),\hdots,f_s^{r-1}(v_s) \} = \{ 1,x,x^2,\hdots,x^{r-1} \}$ and these elements generate $E_s$ whence $S_s$ is completely reachable. The $R$-module morphism $\phi : E_s \rTo R^r$ defined by $\phi(e)=(c_s(e),c_s(f_s(e)),\hdots,c(f^{r-1}(e)))$ is determined by the images
\[
\phi(x^i) = (s_i,s_{i+1},\hdots,s_{i+r-1}) \]
for $0 \leq i < r$ and as these $x^i$ form an $R$-basis for $E_s$, the map $\phi$ is injective, or equivalently that $S_s$ is completely observable if and only if $det(H_r(s)) \not= 0$.

The dual module, $E_s^* = R \epsilon_0 \oplus \hdots \oplus R \epsilon_{r-1}$ where $\epsilon_i(x^j) = \delta_{ij}$. With respect to this basis we have $f_s^*(\epsilon_i) = \epsilon_{i-1}+a_{r-i} \epsilon_{r-1}$ for $i \geq 1$ and $f_s^*(\epsilon_0)=a_r \epsilon_{r-1}$, that is
\[
M_{f_s^*} = \begin{bmatrix} 0 & 1 & & 0 \\
\vdots & & \ddots &  \\
0 & 0 & & 1 \\
a_r & a_{r-1} & \hdots & a_1 \end{bmatrix}  \qquad c_s^* = \begin{bmatrix} s_0 \\ s_1 \\ \vdots \\ s_{r-1} \end{bmatrix} \]
and $v_s^* = (1,0,\hdots,0)$. It follows that $\{ c_s^*,f_s^*(c_s^*),f_s^{* 2}(c_s^*),\hdots,f_s^{* n}(c_s^*) \}$ generate $E_s^*$ if and only if $H_r(s) \in GL_r(R)$.
\end{proof}

\vskip 3mm

\begin{example} \label{F1system} Consider the sequence $s=(1,2,3,\hdots )$ which we encountered in our study of the $\F_1$-zeta function. We have
\[
\begin{bmatrix} 1 & 2 \\ 2 & 3 \end{bmatrix} \in GL_2(\Z) \quad \text{and} \quad det~\begin{bmatrix} 1 & 2 & 3 \\ 2 & 3 & 4 \\ 3 & 4 & 5 \end{bmatrix} = 0 \]
leading to the (minimal) recurrence relation $x^2-2x+1=(x-1)^2$. The corresponding system $S_s=(E_s,f_s,v_s,c_s)$ is split and determined by
\[
E_s = \frac{\Z[x]}{(x-1)^2}, \quad f_s = \begin{bmatrix} 0 & -1 \\ 1 & 2 \end{bmatrix}, \quad v_s = \begin{bmatrix} 1 \\ 0 \end{bmatrix}, \quad \text{and} \quad c_s = \begin{bmatrix} 1 & 2 \end{bmatrix} \]
\end{example}

\vskip 3mm

Clearly, if $S=(E,f,v,c)$ is split, it is a canonical system. Over a field $K$ the converse is also true. Note that the difference between canonical and split systems over $R$ is also important for the co-multiplication on $\mathbb{L}(R)$. 

\vskip 3mm

Over a field $K$ every recursive sequence $s=(s_0,s_1,\hdots) \in \mathbb{L}(K)$ has a {\em minimal} canonical realisation, that is, one with the dimension of the state-space $E$ minimal. To find it, start with a recursive relation  $s_n = a_1 s_{n-1} + a_2 s_{n-2} + \hdots + a_r s_{n-r}$ of depth $r$ and form as above the matrix $H_r(s)$ with columns $H_0,H_1,\hdots,H_{r-1}$. Let $t$ be the largest integer such that the columns $H_0,H_1,\hdots,H_{t-1}$ are linearly independent. If $t=r$ then the previous lemma gives a minimal canonical realisation. If $t < r$ then we have unique coefficients $\alpha_i \in K$ such that $H_t = \alpha_1 H_{t-1} + \alpha_2 H_{t-2} + \hdots + \alpha_t H_0$. But then, it follows that
\[
s_n = \alpha_1 s_{n-1} + \alpha_2 s_{n-2} + \hdots + \alpha_t s_{n-t} \]
is a recursive relation for $s$ of minimal depth $t$. Using this recursive relation we can then construct a canonical realisation as in the previous lemma, with. this time a state-space of minimal dimension. Over a Noetherian domain $R$ one always has a canonical realisation (in the weak sense that the state module $E$ does not have to be projective) see \cite[Theorem IV.5.5]{Tannenbaum} and if $R$ is a principal ideal every linear recursive sequence has a minimal canonical realisation, with free state module, see \cite[VI.5.8.iii]{Tannenbaum}.

\vskip 3mm

Over a field $K$ we know that canonical systems $S_K=(E_K,f_K,v_K,c_K)$, with $dim(E_K)=n$ are also classified up to isomorphism by their {\em transfer function}
\[
T_{S_K}(z) = c_K (zI - M_{f_K})^{-1} v_K  = \frac{Y(z)}{X(z)} = \frac{c_{n-1}z^{n-1}+ \hdots + c_1 z + c_0}{z^n+d_{n-1} z^{n-1} + \hdots + d_1 z + d_0} \]
which are strictly proper rational functions of McMillan degree $n$, that is, $deg(Y(z)) < deg(X(z))=n$ (this is immediate from Cramer's rule) and $(Y(z),X(z))=1$, see for example \cite[II.\S 5]{Tannenbaum}.

\begin{proposition}
Let $T(z) = \frac{Y(z)}{X(z)}$ be a strictly proper rational $K$-function with $Y(z),X(z) \in R[z]$, then there is a completely reachable $R$-linear system $S=(E,f,v,c)$ such that $T(z) = c (zI-M_f)^{-1} v$. If $R$ is a principal ideal domain, this can be achieved by a minimal canonical system.
\end{proposition}

\begin{proof}
We can always find an $R$-system $S'=(E',f',v',c')$ with transfer function $T(z) = c'.(zI - M_{f'})^{-1}.v'$, with $E'=R^n$
\[
f' = \begin{bmatrix} 0 & 1 & 0 & \hdots & 0 \\ 0 & 0 & 1 & \hdots & 0 \\
\vdots & \vdots & & \ddots & \\ 0 & 0 & 0 & \hdots & 1 \\ -d_0 & -d_1 & -d_2 & \hdots & -d_{n-1} \end{bmatrix} \quad v' = \begin{bmatrix} 0 \\ 0 \\ \vdots \\ 0 \\ 1 \end{bmatrix} \quad c' = \begin{bmatrix} c_0 & c_1 & \hdots & c_{n-1} \end{bmatrix}  \]
and this system is completely reachable as $\{ v',f'(v'),f^{'2}(v'),\hdots \}$ generate $R^n$. However, it need not be canonical in general. Still, we can consider its input-output sequence
\[
\mathbat_R(S') = (c'.v',c'.M_{f'}.v',c'.M_{f'}^2.v',\hdots ) \in \mathbb{L}(R) \]
By surjectivity on canonical systems in case $R$ is a principal ideal domain, there is a canonical $R$-system $S=(E,f,v,c)$ with $\mathbat_R(S) = \mathbat_R(S')$, that is,
\[
c'.v' = c.v,~c'.M_{f'}.v'=c.M_f.v,~c'.M_{f'}^2.v' = c.M_f^2.v,~\hdots \]
But, as $T(z) = c'.(zI-M_{f'})^{-1}.v' = c'.v' z^{-1} + c'.M_{f'}.v' z^{-2} + c'.M_{f'}^2.v'.z^{-3} + \hdots $ we see that $T(z)$ is also the transfer function of the canonical $R$-system $S$, proving the claim.
\end{proof}

\begin{definition} For a cyclotomic Bost-Connes datum $\Sigma$, let $\mathcal{S}^{cr}_{\Sigma,R}$ be the full subcategory of $\mathcal{S}_R$ consisting of all completely reachable systems $S=(E,f,v,c)$ such that all zeroes and poles of the transfer function
\[
T_S(z) = c.(zI-M_f)^{-1}.v \]
are in $\Sigma$.
\end{definition}

\begin{example} Continuing example~\ref{F1system}, we have for $T_{S_s}$
\[
\begin{bmatrix} 1 & 2 \end{bmatrix} \begin{bmatrix} z & 1 \\ -1 & z-2 \end{bmatrix}^{-1} \begin{bmatrix} 1 \\ 0 \end{bmatrix} = \frac{z}{(z-1)^2} = Li_{-1} \]
\end{example}

\vskip 3mm

\subsection{Zeta polynomials}
An interesting class of strictly proper rational functions is associated to Manin's 'zeta polynomials' introduced in \cite[\S 1]{ManinZeta} and generalized in \cite{JinZeta} and \cite{OnoZeta}, see also \cite[\S 2.5]{ManinMarcolli}. The terminology comes from a result of F. Rodriguez-Villegas \cite{Rodriguez}. Let $U(z)$ be a polynomial of degree $e$ with $U(1) \not= 0$ and consider the strictly proper rational function
\[
P(z) = \frac{U(z)}{(1-z)^{e+1}} \]
There is a polynomial $H(z)$ of degree $e$ such that the power series expansion of $P(z)$ is
\[
P(z) = \sum_{n=0}^{\infty} H(n) z^n \]
If all roots of $U(z)$ lie on the unit circle, Rodriguez-Villegas proved that the polynomial $Z(z)=H(-z)$ has zeta-like properties: all roots of $Z(z)$ lie on the vertical line $Re(z)=\frac{1}{2}$ and if all coefficients of $U(z)$ are real then $Z(z)$ satisfies the functional equation
\[
Z(1-z) = (-1)^e Z(s) \]
In \cite[\S 1]{ManinZeta} Yuri I. Manin associates such  a zeta-polynomial to each cusp $f$ form of $\Gamma = PSL_2(\Z)$ which is an eigenform for all Hecke operators, and views this polynomial as 'the local zeta factor in characteristic one'. The corresponding numerator $U_f(z)$ of the strictly proper rational function comes from the period polynomial divided by the real zeroes and by \cite{Conrey} the remaining zeros all lie on the unit circle.

In \cite{JinZeta} this construction was generalised to the case of cusp newforms of even weight for the congruence subgroups $\Gamma_0(N)$, where this time the zeroes of period polynomials all lie on the circle with radius $\frac{1}{\sqrt{N}}$. 

\vskip 3mm

Let $Z_i(z)$ be a suitable collection of zeta-polynomials determined by strictly proper rational functions $P_i(z)=\tfrac{U_i(z)}{(1-z)^{d_i}}$ with $U_i(z) \in R[z]$ then we can view the sub bi-ring of $\mathbb{L}(\Z)$ generated by the elements $\mathbat_{R}(S_i) \in \mathbb{L}(R)$, where $S_i$ is a completely reachable or minimal canonical system realizing $P_i(z)$, as a representative for the collection of zeta-polynomials in the $\mathbf{comm^+_{bi}}$-version of $\F_1$- geometry. Again, we can define similarly versions relative to a cyclotomic Bost-Connes datum $\Sigma$ by imposing that the zeroes of the zeta-polynomials must lie in $\Sigma$.


\begin{thebibliography}{10}

\bibitem{Almkvist}
G. Almkvist, {\it Endomorphisms of finitely generated projective modules over a commutative ring}, Arkiv f\"ur Matematik Volume 11, Numbers 1-2 (1973), 263 - 301.

\bibitem{Borger1}
J. Borger, {\it The basic geometry of Witt vectors, 1: the affine case}, J. Algebra and Number Theory 5 (2011) 231-285

\bibitem{Borger2}
J. Borger, {\it The basic geometry of Witt vectors, 2: spaces}, Math. Ann. 351 (2011) 877-933


\bibitem{ConnesConsaniBC}
Alain Connes and Caterina Consani, {\it On the arithmetic of the BC-system}, {\tt arXiv:1103.4672} (2011)


\bibitem{Conrey}
J. B. Conrey, D.W. Farmer and O. Imamoglu, {\em The nontrivial zeros of period polynomials lie on the unit circle}, Int. Math. Res. Not. 20 (2013) 4758-4771, {\tt arXiv:1201.2322}



\bibitem{GilletSoule}
H. Gillet and C. Soul\'e, {\em Descent, motives and $K$-theory}, J. Reine Angew. Mat. 478 (1996) 127-176

\bibitem{Gorsky}
E. Gorsky, {\em Adams operations and power structures}, Moscow Math. J. 9 (2009) 305-323 {\tt arXiv:0803.3118} (2008)

\bibitem{Grayson}
D. Grayson, {\em $SK_1$ of an interesting principal ideal domain}, Jourrnal of Pure and Appl. Alg. 20 (1981) 157-163

\bibitem{Hazewinkel}
Michiel Hazewinkel, {\em Witt vectors, Part 1}. {\tt arXiv:0804.3888} (2008)



\bibitem{JinZeta}
S. Jin, W. Ma,K. Ono and K. Soundararajan, {\em The Riemann hypothesis for period polynomials of Hecke eigenforms}, Proc. Nat. Ac. Sci. USA, 113 (2016) 2603-2608, {\tt arXiv:1601.03114}

\bibitem{Knutson}
Donald Knutson, {\em $\lambda$-rings and the representation theory of the symmetric group}, Springer LNM 308 (1973) Springer-Verlag

\bibitem{LarsenLuntz}
M. Larsen and V. Luntz, {\em Rationality criteria for motivic zeta functions}, Compos. Math. 140 (2004) 1537-1560

\bibitem{LarsonTaft}
Richard G. Larson and Earl J. Taft, {\em The algebraic structure of linearly recursive sequences under Hadamard product}

\bibitem{LBrecursive}
Lieven Le Bruyn, {\it Linear recursive sequences and $\mathbf{Spec}(\Z)$ over $\mathbb{F}_1$},  
Communications in Algebra  45:7 (2017) p. 3150-3158


\bibitem{LieberManinMarcolli}
Joshua Lieber, Yuri I. Manin and Matilde Marcolli, {\em Bost-Connes systems and $\F_1$-structures in Grothendieck rings, spectra, and Nori motives}, {\tt arXiv:1901.00020} (2019)

\bibitem{LiuSebag}
Q. Liu and J. Sebag, {\em The Grothendieck ring of varieties and piecewise isomorphisms}, Math. Zeit. 265 (2010) 321-354

\bibitem{LopezLorscheid}
J. Lopez Pena and O. Lorscheid, {\em Torified varieties and their geometries over $\F_1$}, Math. Z. 267 (2011) 605-643, {\tt arXiv:0903.2173} (2009)

\bibitem{ManinCyclotomy}
Yuri I. Manin, {\it Cyclotomy and analytic geometry over $\mathbb{F}_1$}, {\tt arXiv:0809.1564} (2008), in "Quanta of Maths", conference in honor of Alain Connes. Clay Math. Proc. 11 (2011) 385-408

\bibitem{ManinNumbers}
Yuri I. Manin, {\em Numbers as functions}, {\tt arXiv:1312.5160} (2013)

\bibitem{ManinZeta2}
Yuri I. Manin, {\em Lectures on zeta functions and motives (according to Deninger and Kurokawa)}, In: Columbia University Number Theory Seminar (1992), Ast\'erisque 228 (1995) 121-164

\bibitem{ManinZeta}
Yuri I. Manin, {\em Local zeta functions and geometries under $\mathbf{Spec}(\Z)$}, {\tt arXiv:1407.4969} (2014)

\bibitem{ManinMarcolli0} 
Yuri I. Manin and Matilde Marcolli, {\em Moduli operad over $\F_1$}, in 'Absolute Arithmetic and $\F_1$-Geometry', EMS (2016) 331-361, {\tt arXiv:1302.6526} (2013)


\bibitem{ManinMarcolli}
Yuri I. Manin and Matilde Marcolli, {\it Homotopy types and geometries below $\mathbf{Spec}(\Z)$}, {\tt arXiv:1806.10801} (2018)

\bibitem{MarTab}
Matilde Marcolli and Goncalo Tabuada, {\it Bost-Connes systems, categorification, quantum statistical mechanics, and Weil numbers}, {\tt arXiv:1411.3223} (2014)


\bibitem{Mezzetti}
G. Mezzetti, {\tt halloweenmath}, CTAN {\tt https://ctan.org/pkg/halloweenmath?lang=en} (2017)

\bibitem{Naumann}
N. Naumann, {\em Algebraic independence in the Grothendieck ring of varieties}, Trans. AMS 359 (2007) 1653-1683

\bibitem{OnoZeta}
Ken Ono, Larry Rolen and Florian Sprung, {\em Zeta-polynomials for modular form periods}, {\tt arXiv:1602.00752} (2016)

\bibitem{PetersonTaft}
B. Peterson and E.J. Taft, {\em The Hopf algebra of linear recursive sequences}, Aequationes Math. 20 (1980) 1-17


\bibitem{Ramachandran}
Niranjan Ramachandran, {\em Zeta functions, Grothendieck groups, and the Witt ring}, Bull. Sci. Math. Soc. Math. France 139 (2015) 599-627, {\tt arXiv:1407.1813} (2014)

\bibitem{RamaTabu}
Niranjan Ramachandran and Gancalo Tabuada, {\em Exponentiable motivic measures}, J. Ramanujan Math. Soc. 30 (2015) 349-360, {\tt arXiv:1412.1795} (2014)

\bibitem{Rodriguez}
Fernando Rodriguez-Villegas, {\em On the zeros of certain polynomials}, Proc. AMS 130 (2002) 2251-2254

\bibitem{Tannenbaum}
Allen Tannenbaum, {\it Invariance and System Theory: Algebraic and Geometric Aspects}, Springer LNM 845 (1981)


\bibitem{Wilkerson}
C. Wilkerson, {\em Lambda-rings, binomial domains and vector bundles over $CP(\infty)$}, Comm. Alg. 10 (1982) 311-328

\end{thebibliography}
\end{document}